\documentclass[10pt]{amsart}
\usepackage{amsmath,amsthm,amsfonts,amssymb,latexsym,enumerate,url,hyperref} 

\usepackage[mathscr]{eucal}
\usepackage{mathrsfs}
\usepackage{rotating}
\usepackage{stmaryrd}
\usepackage{xcolor}

\swapnumbers

\newtheorem{thm}{Theorem}[section]
\newtheorem{lem}[thm]{Lemma}

\newtheorem{prop}[thm]{Proposition}
\newtheorem{cor}[thm]{Corollary}

\newtheorem*{thmB}{THEOREM B}
\newtheorem*{corC}{COROLLARY C}
\newtheorem*{conjA}{CONJECTURE A}

\newtheorem*{conj0}{CONJECTURE}

\theoremstyle{remark}

\newtheorem{rem}[thm]{Remark}
\newtheorem{examples}[thm]{Examples}

\numberwithin{equation}{section}

\renewcommand{\le}{\leqslant}
\renewcommand{\ge}{\geqslant} 

\newcommand{\BIGOP}[1]
  {\mathop{\mathchoice
  {\raise-0.22em\hbox{\huge $#1$}}
  {\raise-0.05em\hbox{\Large $#1$}}{\hbox{\large $#1$}}{#1}}}
\newcommand{\calM}{\mathcal{M}}
\newcommand{\calMtilde}{\tilde{\calM}}

\newcommand{\calR}{\mathcal{R}}

\newcommand{\CC}{\mathbb{C}}

\newcommand{\chireg}{\chi_{\mathrm{reg}}}
\newcommand{\Gal}{\mathrm{Gal}}

\newcommand{\GL}{\mathrm{GL}}

\newcommand{\id}{\mathrm{id}}
\newcommand{\Irr}{\mathrm{Irr}}
\newcommand{\lexp}[2]{\setbox0=\hbox{$#2$} \setbox1=\vbox to
                 \ht0{}\,\box1^{#1}\!#2}
\newcommand{\Max}{\mathrm{Max}}

\newcommand{\NN}{\mathbb{N}}

\newcommand{\QQ}{\mathbb{Q}}
\newcommand{\scrX}{\mathscr{X}}

\newcommand{\Tr}{\mathrm{Tr}}
\newcommand{\und}[1]{\underline{#1}}

\newcommand{\ZZ}{\mathbb{Z}}

\title{Feit's Conjecture, the canonical Brauer induction formula, and Adams operations}

\author{Robert Boltje}
\address{Department of Mathematics, University of California Santa Cruz, CA 95064, USA}
\email{boltje@ucsc.edu}

\author{Gabriel Navarro}
\address{Departament de Matem\`atiques, Universitat de Val\`encia, 46100 Burjassot,
Val\`encia, Spain}
\email{gabriel@uv.es}

\thanks{Research  supported by Ministerio de Ciencia e Innovaci\'on
PID2019-103854GB-I00 and CDEIGENT grant CIDEIG/2022/29 funded by Generalitat Valenciana}



\date{\today}

\begin{document}

\sloppy


\maketitle




\begin{abstract}
This paper is motivated by a strong version of Feit's conjecture, first formulated by the authors in joint work with A. Kleshchev and P. H. Tiep in 2025, concerning the conductor $c(\chi)$ of an irreducible character $\chi$ of a finite group $G$. We connect the conjecture with the following construction: For any positive integer $n$ dividing the exponent of $G$ and for any character $\chi$ of $G$, we introduce an integer-valued invariant $S(G,\chi,n)$ which can be defined as the sum of certain coefficients of the canonical Brauer induction formula of $\chi$, or alternatively as the multiplicity of the trivial character in a specified integral linear combination of Adams operations of $\chi$. We show two facts about this invariant. The first seems of independent interest (apart from Feit's conjecture): $S(G,\chi,n)$ is always non-negative, and it is positive if and only if a representation affording $\chi$ involves an eigenvalue of order $n$. Secondly, the strong version of Feit's conjecture holds for an irreducible character $\chi$ if and only if $S(G,\chi, c(\chi))>0$.
\end{abstract}


\section{Introduction}\label{sec intro}

In \cite[3.2(b)]{Feit1980}, Walter Feit posed a question that has since become known as Feit’s conjecture. This conjecture would imply Brauer’s Problem 41, as well as several other earlier results in the field.

\medskip
\begin{conj0}[Feit 1980]
Let $G$ be a finite group and $n$ a positive integer. If $G$ has an irreducible character of conductor $n$, 
then $G$ has an element of order $n$.
\end{conj0}

Here, for an abelian Galois extension $K$ of the rational numbers $\QQ$, the {\em conductor} $c(K)$ of $K$ 
is the smallest positive integer $n$ such that $K$ is contained in $\QQ_n$, the $n$-th cyclotomic field.
For any character $\chi$ of a finite group $G$, the field $\QQ(\chi)$ obtained by adjoining the character values $\chi(x)$, $x\in G$, to $\QQ$,
is an abelian Galois extension of $\QQ$. One calls $c(\chi):=c(\QQ(\chi))$ the {\em conductor of $\chi$}.

In the recent paper \cite{BKNT}, a reduction of Feit's conjecture to the so-called {\em inductive Feit condition}
for all non-abelian simple groups, was given. In the same paper, the following conjecture was stated. 

\begin{conjA}[Boltje-Kleshchev-Navarro-Tiep 2024] Let $G$ be a finite group and $\chi$ an irreducible character of $G$. 
Then there exists a subgroup $H$ of $G$ and a linear character $\varphi$ of $H$ 
such that $\varphi$ is a constituent of the restriction $\chi_H$ and $c(\chi)=o(\varphi)$, the order of $\varphi$.
\end{conjA}

Note that Conjecture $A$ is equivalent to the seemingly weaker {\bf Conjecture A'} in which the condition $c(\chi)=o(\varphi)$ is replaced with $c(\chi)\mid o(\varphi)$. In fact, if $c(\chi)\mid o(\varphi)$, then let $g\in G$ be such that the coset of $g$ generates the cyclic group $H/\ker(\varphi)$. Then the restriction $\psi$ of $\varphi$ to some subgroup of $\langle g\rangle$ will satisfy $c(\chi)=o(\psi)$. Also, some power of the element $g$ will have order $c(\chi)$ which shows that Conjecture~A' implies Feit's conjecture. Thus, if Conjecture~A (or equivalently Conjecture A') holds for given $G$ and $\chi$,
then also Feit's conjecture holds for this $G$ and $\chi$.


\smallskip
Conjecture~A was proved in \cite{BKNT} for all solvable groups. 
It also holds for all sporadic simple groups, as can be verified with \cite{GAP}, although it has not yet been reduced to a condition on simple groups.

\smallskip
When studying the plausibility of Conjecture A, one soon arrives at the
basic question:  how can $\chi$ determine a pair $(H,\varphi)$? One of the essential and most famous theorem in character theory is Brauer's induction theorem that asserts that every irreducible character can be expressed as an integer   combination of the induction of linear characters of subgroups of $G$. As proved by V. Snaith in 1986 by topological methods
(\cite{Snaith88}), and the first author in 1990 by algebraic methods (\cite{Boltje1990}), this can be achieved in a canonical way; that is, there is 
 a canonical expression  
 $$\chi=\sum_{[H,\varphi]_G} \alpha_{[H,\varphi]_G} \varphi^G \, ,$$
where here the right hand side of the above equation is an element of the free abelian group on the set of $G$-conjugacy classes $[H,\varphi]_G$ 
of pairs $(H,\varphi)$, where $H$ is a subgroup of $G$ and $\varphi$ is a linear character of $H$, and the $\alpha_{[H,\varphi]_G}$ are uniquely determined integers.  The first motivation for this paper was the realization that if $\chi$ is an irreducible character of $G$ 
and there exists a pair $(H,\varphi)$ satisfying the condition in Conjecture~A', 
then there must also exist such a pair that occurs in the canonical Brauer induction formula of $\chi$ (see Proposition~\ref{prop 5 equivalences}). This statement is the reason to consider the slight reformulation of Conjecture A to Conjecture A'.
 
 \smallskip
 This led us to introduce a new invariant.  For any positive integer $n$ and any character $\chi$ of $G$, we define the sum
 \begin{equation}\label{eqn S def}
   S(G,\chi,n):=\sum_{\substack{[H,\varphi]_G\\ n\mid o(\varphi)}} \alpha_{[H,\varphi]_G}(\chi)\,.
\end{equation}

 In our first main result, we prove the following properties of the integer $S(G,\chi,n)$ using the behavior of the canonical induction formula with respect to Adams operations.

\begin{thmB}
Let $G$ be a finite group, let $\chi$ be a character of $G$ afforded by a representation $\scrX\colon G\to\GL_d(\CC)$, and let $n$ be a positive divisor of $\exp(G)$.

\smallskip
{\rm (a)} The integer $S(G,\chi,n)$ is equal to the multiplicity of the trivial character in the virtual character
\begin{equation}\label{eqn alternating Adams character}
   \sum_{\rho\subseteq\pi(n)} (-1)^{|\rho|} \Psi^{n(\rho)}(\chi)
\end{equation}
of $G$.

\smallskip
{\rm (b)} One has $S(G,\chi,n)\ge 0$.

\smallskip
{\rm (c)} The following are equivalent:

\smallskip
\qquad {\rm (i)} There exists a subgroup $H$ of $G$ and a linear constituent $\varphi$ of $\chi_H$ such that $n = o(\varphi)$.

\smallskip
\qquad {\rm (ii)} One has $S(G,\chi,n)> 0$.

\smallskip
\qquad {\rm (iii)} There exists an element $x\in G$ and an eigenvalue $\zeta$ of $\scrX(x)$ such that $n = o(\zeta)$.
\end{thmB}

Here, in Part~(a), $\pi(n)$ is the set of prime divisors of $n$ and $n(\rho)= \exp(G)_{\rho'}\cdot\und{n}_\rho$, where $\und{n}:= n/\prod_{p\in\pi(n)} p$, see also the beginning of Section~\ref{sec number theory}. Moreover, $\Psi^n(\chi)$ denotes the $n$-th Adams operation applied to $\chi$, i.e., $\Psi^n(\chi)(g)=\chi(g^n)$, see the beginning of Section~\ref{sec Adams}.

\smallskip
The connection between the definition of $S(G,\chi,n)$ in (\ref{eqn S def}) and the formula in Part~(a) stems from the property (\ref{eqn canind and Adams}) of the canonical Brauer induction formula. Note that the summation in (\ref{eqn canind and Adams}) runs over $G$-orbits of pairs $(H,\varphi)$ satisfying $o(\varphi)\mid n$, while in (\ref{eqn S def}) the condition is $n\mid o(\varphi)$. To relate these two conditions, we need  Lemma~\ref{lem summatory}.

\smallskip
Part~(b) is quite surprising and raises the question if it has a deeper structural reason than just the application of the technical number theoretical Lemma~\ref{lem technical}.

\smallskip
In view of Part (c) of Theorem B, the number $S(G,\chi,n)$ is an indicator for one of the matrices $\scrX(x)$, $x\in G$, to have an eigenvalue of order $n$. 

\smallskip
When $\chi$ is irreducible and $n=c(\chi)$, Theorem B allows us to obtain the following consequence, which shows that the multiplicity of the trivial character in
the virtual character in (\ref{eqn alternating Adams character}) is an indicator for Conjecture~A to hold for $G$ and $\chi$.

\begin{corC}
Let $G$ be a finite group and let $\chi$ be an irreducible character of $G$ afforded by the representation $\scrX\colon G\to \GL_d(\CC)$. 
Then the following are equivalent:

\smallskip
{\rm (i)} Conjecture~A holds for $G$ and $\chi$.

\smallskip
{\rm (ii)} The multiplicity of the trivial character in the virtual character in (\ref{eqn alternating Adams character}) 
of $G$ with $n=c(\chi)$ is positive.

\smallskip
{\rm (iii)} There exists an element $x$ of $G$ and an eigenvalue $\zeta$ of $\scrX(x)$ such that $c(\chi) = o(\zeta)$.
\end{corC}

The paper is arranged as follows. In Section~2 we recall the definition and relevant properties of the canonical Brauer induction formula.
Section~3 establishes some number theoretic results that are necessary to prove Theorem~B. In Section~4 we prove Theorem~B \and compute $S(G,\chi,n)$ for  special values of $n$ and $\chi$.

\bigskip\noindent
{\bf Acknowledgement}~~The first author would like to thank the Mathematics Department of the University of Valencia
for their hospitality during a four-month visit.


\section{The canonical Brauer induction formula}\label{sec ind formula}

In this section we recall the basic properties of the canonical Brauer induction formula from \cite{Boltje1990}.
For a finite group $G$, consider the set
\begin{equation*}
   \calM(G):=\{(H,\varphi)\mid H\le G, \varphi\in \hat H\}\,,
\end{equation*}
where $\hat H$ denotes the multiplicative group of linear characters of $H$. It is a partially order set (poset for short) under
the relation $(K,\psi)\le (H,\varphi)$, defined by $K\le H$ and $\psi=\varphi_K$.
Conjugation $\lexp{g}{(H,\varphi)}:=(\lexp{g}{H},\lexp{g}{\varphi})$ by elements in $G$ defines an action of $G$ on $\calM(G)$ 
via poset automorphisms. Here $\lexp{g}{H}:=gHg^{-1}$ and $(\lexp{g}{\varphi})(\lexp{g}{h}):=\varphi(h)$, for $h\in H$, 
where $\lexp{g}{h}:=ghg^{-1}$. 

\medskip
Denote by $R_+(G)$ the free abelian group with basis given as the set of $G$-orbits $[H,\varphi]_G$ of elements 
$(H,\varphi)$ in $\calM(G)$. There exists a restriction map 
\begin{equation*}
   R_+(G)\to R_+(U)\,,\quad [H,\varphi]_G\mapsto 
   \sum_{g\in[U\backslash G /H]} \bigl[U\cap\lexp{g}{H},(\lexp{g}{\varphi})_{U\cap\lexp{g}{H}}\bigr]_U\,,
\end{equation*}
whenever $U\le G$. Here, $[U\backslash G/H]$ denotes a set of representatives of the $(U,H)$-double cosets in $G$. The map 
\begin{equation*}
   b_G\colon R_+(G)\to R(G)\,, \quad [H,\varphi]_G\mapsto \varphi^G\,,
\end{equation*}
commutes with restrictions. Here, $R(G)$ denotes the character ring of $G$ and $\varphi^G$ the induction of $\varphi$ to $G$.
The {\em canonical Brauer induction formula} is the unique collection of group homomorphisms
\begin{equation*}
   a_G\colon R(G)\to R_+(G)\,,\quad \chi\mapsto\sum_{[H,\varphi]_G\in G\backslash\calM(G)}\alpha_{[H,\varphi]_G}(\chi)\, [H,\varphi]_G\,,
\end{equation*}
where $G$ runs through all finite groups, which commutes with restrictions to subgroups and satisfies the normalization condition 
that $\pi_G\circ a_G = p_G$, where $p_G\colon R(G)\to \ZZ \hat G$ and $\pi_G\colon R_+(G)\to \ZZ \hat G$ are the natural projection maps. 
The map $a_G$ is a section of the map $b_G$ from above, i.e., $b_G\circ a_G=\id_{R(G)}$. In other words,
\begin{equation*}
   \chi=\sum_{[H,\varphi]_G} \alpha_{[H,\varphi]_G}(\chi)\, \varphi^G\,,
\end{equation*}
for all $\chi\in R(G)$.

\medskip
One has the explicit formula
\begin{equation}\label{eqn a_G 1}
   a_G(\chi)=\frac{1}{|G|} \sum_{(H_0,\varphi_0)<\cdots<(H_n,\varphi_n)} (-1)^n |H_0|(\chi_{H_n},\varphi_n) [H_0,\varphi_0]_G\,,
\end{equation}
for all $\chi\in R(G)$, where the sum runs over all chains in the poset $\calM(G)$. Here, $(-,-)$ denotes the Schur inner product on $R(G)$. Another explicit formula is 
\begin{equation}\label{eqn a_G 2}
   a_G(\chi) =  \sum_{\substack{(H_0,\varphi_0)<\cdots<(H_n,\varphi_n)\\ {\rm mod}\ G}} (-1)^n (\chi_{H_n},\varphi_n) [H_0,\varphi_0]_G\,,
\end{equation}
where the sum this time runs only over a set of representatives of the $G$-orbits of chains in $\calM(G)$. 
That the two formulas actually give the same element in $R_+(G)$ is non-trivial and one of the main theorems of \cite{Boltje1994}. 
Since $b_G\circ a_G=\id_{R(G)}$, we have
\begin{equation*}
   \chi=\sum_{\substack{(H_0,\varphi_0)<\cdots<(H_n,\varphi_n)\\ {\rm mod}\ G}} (-1)^n (\chi_{H_n},\varphi_n)\, \varphi_0^G
\end{equation*}
for all $\chi\in R(G)$.

\medskip
Next we'll derive a connection between Conjecture~A and the the canonical Brauer induction formula. 
For any character $\chi$ of $G$, we set
\begin{equation*}
   \calM(G,\chi):=\{(H,\varphi)\mid (\chi_H,\varphi)>0\}\,.
\end{equation*}
Note that if $(H,\varphi)\in\calM(G,\chi)$ and $(K,\psi)\le (H,\varphi)$ then also 
$(K,\psi)\in\calM(G,\chi)$ and note that $\calM(G,\chi)$ is closed under conjugation. 
We denote by $\Max(\calM(G,\chi))$ the maximal elements of the poset $\calM(G,\chi)$.

Further, for any character $\chi$ of $G$, we set
\begin{equation*}
   \calMtilde(G,\chi):=\{(H,\varphi)\mid \alpha_{[H,\varphi]_G}(\chi)\neq 0\}
\end{equation*}
and denote by $\Max(\calMtilde(G,\chi))$ the maximal elements of the poset $\calMtilde(G,\chi)$. 
By the explicit formula~(\ref{eqn a_G 1}),
one has $\calMtilde(G,\chi)\subseteq \calM(G,\chi)$ (a strict inclusion in general).

\begin{prop}\label{prop maximal Brauer}
For every finite group $G$ and every character $\chi$ of $G$, one has $\Max(\calM(G,\chi))=\Max(\calMtilde(G,\chi))$.
\end{prop}

\begin{proof}
This follows again from the explicit formula (\ref{eqn a_G 1}). 
In fact, if $(H,\varphi)\in\Max(\calM(G,\chi))$ then the only chains contributing to the basis element 
$[H,\varphi]_G$ in $a_G(\chi)$ are the chains of length $0$ of the form $\lexp{g}{(H,\varphi)}$ for $g\in G$. 
Thus, $\alpha_{[H,\varphi]_G}(\chi)>0$ and $(H,\varphi)\in\calMtilde(G,\chi)$. 
Since $\calMtilde(G,\chi)\subseteq\calM(G,\chi)$ we also obtain $(H,\varphi)\in\Max(\calMtilde(G,\chi))$. 
Conversely, let $(H,\varphi)\in\Max(\calMtilde(G,\chi))$. 
Then $(H,\varphi)\in\calM(G,\chi)$ and there exists $(H',\varphi')\in\Max(\calM(G,\chi))$ with $(H,\varphi)\le (H',\varphi')$. 
By the first part of the proof, $(H',\varphi')\in\Max(\calMtilde(G,\chi))\subseteq\calMtilde(G,\chi)$ 
so that $(H,\varphi)=(H',\varphi')\in\Max(\calM(G,\chi))$, by the maximality of $(H,\varphi)$ in $\calMtilde(G,\chi)$.
\end{proof}

\begin{prop}\label{prop 5 equivalences}
Let $G$ be a finite group, let $\chi$ be a character of $G$, and let $n$ be a positive integer. The following are equivalent:

\smallskip
{\rm (i)} There exists an element $(H,\varphi)$ in $\calM(G,\chi)$ with $n\mid o(\varphi)$.

\smallskip
{\rm (ii)} There exists an element $(H,\varphi)$ in $\Max(\calM(G,\chi))$ with $n\mid o(\varphi)$.

\smallskip
{\rm (iii)} There exists an element $(H,\varphi)$ in $\calMtilde(G,\chi)$ with $n\mid o(\varphi)$.

\smallskip
{\rm (iv)} There exists an element $(H,\varphi)$ in $\Max(\calMtilde(G,\chi))$ with $n\mid o(\varphi)$.

\smallskip
{\rm (v)} There exists an element $(H,\varphi)$ in $\calM(G,\chi)$ such that $H$ is cyclic and $n\mid o(\varphi)$.
\end{prop}

\begin{proof}
First note that if $(K,\psi)\le(H,\varphi)$ then $o(\psi)\mid o(\varphi)$. This implies (i)$\iff$(ii) and (iii)$\iff$(iv). 
Moreover, Proposition~\ref{prop maximal Brauer} implies (ii)$\iff$(iv). Clearly (v) implies (i). 
To see that (i) implies (v), let $(H,\varphi)\in\calM(G,\chi)$ with $n\mid o(\varphi)$. 
Choose $h\in H$ such that $h\ker(\varphi)$ generates $H/\ker(\varphi)$ and set 
$K:=\langle h\rangle$ and $\psi:=\varphi_K$. Then $(K,\psi)$ satisfies (v).
\end{proof}

\begin{rem}\label{rem reformulation} Statement (v) in Proposition~\ref{prop 5 equivalences} is equivalent to 

\smallskip
(v') {\it There exists and element $(H,\varphi)$ in $\calM(G,\chi)$ such that $H$ is cyclic and $n=o(\varphi)$.}

\smallskip
In fact if $(H,\varphi)$ satisfies $n\mid o(\varphi))$, then some $(K,\psi)\le (H,\varphi)$ satisfies $n= o(\psi)$. This also shows that Statement (i) in Proposition~\ref{prop 5 equivalences} is equivalent to

\smallskip
(i') {\it There exists an element $(H,\varphi)\in\calM(G,\chi)$ with $n=o(\varphi)$.}
\end{rem}

Of course the above Proposition is of particular interest to us in the case that $\chi\in\Irr(G)$ and $n=c(\chi)$.
In this situation, it implies that if there exists a pair $(H,\varphi)\in\calM(G)$ satisfying Conjecture~A, 
then there also exists such a pair occurring with non-zero coefficient in $a_G(\chi)$. 
One can even choose $(H,\varphi)$ to be in $\Max(\calMtilde(G,\chi))$. 
In Section~\ref{sec Adams} we will study, for any character $\chi$ of $G$ and positive divisor $n$ of $\exp(G)$, the sum 
\begin{equation}\label{eqn S function}
   S(G,\chi,n):=\sum_{\substack{[H,\varphi]_G\in G\backslash\calM(G) \\ n \mid o(\varphi)}} \alpha_{[H,\varphi]_G}(\chi)\,.
\end{equation}
Again, if $\chi\in\Irr(G)$, $n=c(\chi)$ and $S(G,\chi,c(\chi))\neq 0$, then at least one of the summands needs to be non-zero, so
that there exists $(H,\varphi)\in\calMtilde(G,\chi)$ with $c(\chi)\mid o(\varphi)$ and Conjecture~A holds for $G$ and $\chi$.

\smallskip
Before we continue this line of thought we need to establish a few lemmas from number theory.


\section{Auxiliary results from number theory}\label{sec number theory}

In this section we prove three lemmas in number theory that will be used in Section~\ref{sec Adams}.

\medskip
If $\pi$ is a set of primes we denote by $\pi'$ the complementary set of primes. 
Each positive integer $n$ has a decomposition $n=n_\pi \cdot n_{\pi'}$ in to a $\pi$-part $n_\pi$ and a $\pi'$-part $n_{\pi'}$. 
We denote by $\pi(n)$ the set of prime divisors of $n$ and set $\und n:=n/\prod_{p\in\pi(n)} p$.


\smallskip
In the situation of the following Lemma, it is clear that $f^+$ can be expressed, as a double sum over elements in $\Lambda$, in terms of $f_+$, since $f$ can be expressed via M\"obius inversion in terms of $f_+$. The point of the lemma is that this double sum can be simplified to a single sum running over certain sets of primes.

\begin{lem}\label{lem summatory}
Let $N$ be a positive integer, let $\Lambda$ denote the set of positive divisors of $N$, and let 
$f\colon \Lambda\to A$ be a function with values in an additive abelian group $A$. 
Let $f_+\colon \Lambda\to A$ and $f^+\colon \Lambda\to A$ be the functions defined by
\begin{equation*}
   f_+(n):= \sum_{\substack{d\in\Lambda\\ d\mid n}} f(d)\quad\text{and}\quad
   f^+(n):=\sum_{\substack{d\in\Lambda\\ n\mid d}} f(d)\,.
\end{equation*}
Then
\begin{equation*}
   f^+(n) = \sum_{\rho\subseteq \pi(n)} (-1)^{|\rho|} f_+(n(\rho))\,,
\end{equation*}
for all $n\in\Lambda$, where $n(\rho):=\und n_{\rho} N_{\rho'}$.
\end{lem}
   
\begin{proof}
Expanding $f_+$ and $f^+$ by their definitions, we need to show that
\begin{equation}\label{eqn expanded}
   \sum_{\substack{d\in\Lambda\\ n\mid d}} f(d) = \sum_{\rho\subseteq\pi(n)} (-1)^{|\rho|} \sum_{\substack{d\in\Lambda\\ d\mid n(\rho)}} f(d)\,,
\end{equation}
for every $n\in\Lambda$. Since both sides in the above equation are additive in $f$, it suffices to show this for 
a function $f$ with takes a non-zero value on precisely one element $e\in\Lambda$. We distinguish two cases. 

If $n\mid e$ then the left hand side of Equation~(\ref{eqn expanded}) is equal to $f(e)$ and its right hand is equal to 
\begin{equation*}
   \Bigl(\sum_{\substack{\rho\subseteq\pi(n)\\ e\mid n(\rho)}} (-1)^{|\rho|}\Bigr) f(e)\,.
\end{equation*}
But $\rho=\emptyset$ is the only subset $\rho$ of $\pi(n)$ with $n\mid e$ and $e\mid n(\rho)$, 
so that also the right hand side of Equation~(\ref{eqn expanded}) is equal to $f(e)$.

If $n\nmid e$ then the left hand side of Equation~(\ref{eqn expanded}) is equal to $0$, and we will need to show that
\begin{equation}\label{eqn reduced}
   \sum_{\substack{\rho\subseteq \pi(n)\\ e\mid n(\rho)}} (-1)^{|\rho|} = 0\,.
\end{equation}
But for any $\rho\subseteq \pi(n)$ one has $e\mid n(\rho)$ if and only if $\rho$ is contained in the set
\begin{equation*}
   \rho_0:=\{p\in\pi(n)\mid v_p(e) \le v_p(n)-1\}\,.
\end{equation*}
Here, $v_p(n)$ denotes the $p$-value of the integer $n$. Thus it suffices to show that $\sum_{\rho\subseteq\rho_0} (-1)^{|\rho|} = 0$.
But this follows immediately from $\rho_0\neq \emptyset$ (since $n\nmid e$) and the binomial formula for $(1-1)^{|\rho_0|}$.
\end{proof}

Recall that for any finite Galois extension $L/K$, the {\em relative trace map} $\Tr_{L/K}\colon L\to K$ is defined by 
$\Tr_{L/K}(a):=\sum_{\sigma\in\Gal(L/K)} \sigma(a)$, for $a\in L$. If $f(x)$ is the minimal polynomial of $a$ over $K$ and has degree $n$ then $-Tr_{L/K}(a)$ is equal to the coefficient of $x^{n-1}$ in $f(x)$. If $L=\QQ_n$, the $n$-th cyclotomic field, and if $\zeta$ is a root of unity in $\QQ_n$, then $-\Tr_{\QQ_n/\QQ}(\zeta)=
\sum_{k\in(\ZZ/n\ZZ)^\times} \zeta^k$. Moreover, if $\zeta$ has order $n$ then $-\Tr_{\QQ_n/\QQ}(\zeta)$ is equal to the coefficient at $x^{\phi(n)-1}$ in the cyclotomic polynomial $\Phi_n(x)$, the minimal polynomial of $\zeta$, where  $\phi\colon \NN\to\NN$ denotes the {\em Euler totient function} given by $\phi(n)=|(\ZZ/n\ZZ)^\times|$.

Also recall that the number theoretic Möbius function $\mu$ is defined on the set of positive integers by
$\mu(n)=0$ if $n$ is not square-free and $\mu(n)=(-1)^r$ if $n$ is the product of $r$ pairwise distinct primes. Equivalently, $\mu$ is determined recursively by the equations $\mu(1)=1$ and $\sum_{1\le d\mid n}\mu(d)=0$ for every integer $n>1$.

\begin{lem}\label{lem trace}
Let $n\in\NN$ and let $\zeta$ be a root of unity in $\QQ_n$. Then 
\begin{equation*}
   \Tr_{\QQ_n/\QQ}(\zeta) = \mu(o(\zeta))\frac{\phi(n)}{\phi(o(\zeta))}\,,
\end{equation*} 
where $o(\zeta)$ denotes the multiplicative order of $\zeta$, $\mu$ denotes the number theoretic M\"obius function, and $\phi$ denotes Euler's totient function.
\end{lem}

\begin{proof}
Set $k:=o(\zeta)$. By the transitivity of the trace map we have 
$\Tr_{\QQ_n/\QQ}(\zeta)=[\QQ_n\colon\QQ_k]\cdot\Tr_{\QQ_k/\QQ}(\zeta)=\frac{\phi(n)}{\phi(k)}\Tr_{\QQ_k/\QQ}(\zeta)$. 
We claim that $\Tr_{\QQ_k/\QQ}(\zeta)= \mu(k)$ which immediately implies the result. 
The cyclotomic polynomials satisfy the equation
\begin{equation}\label{eqn cyclotomic}
   x^k-1 = \prod_{1\le d\mid k} \Phi_d(x)\,.
\end{equation}
For any positive integer $d$ define the integer $c_d$ such that $-c_d$ is the coefficient of $x^{\phi(d)-1}$ in $\Phi_d(x)$, or equivalently by $c_d:=\Tr_{\QQ_d/\QQ}(\eta)$, where $\eta$ is a root of unity of order $d$. Then clearly $c_1=1$ and Equation~(\ref{eqn cyclotomic}) implies that $0=c_k=\sum_{1\le d\mid n}c_d$ when $k>1$. Thus, the numbers $c_d$ satisfy the same recursion formula as the number theoretic M\"obius function $\mu$. 
This proves the claim and completes the proof of the lemma.
\end{proof}

\begin{lem}\label{lem technical}
Let $N$ and $n$ be positive integers with $n\mid N$. 
Furthermore, let $t$ be a positive divisor of $N$ and let $\zeta\in\QQ_t$ be a root of unity of order dividing $N$.
Finally, set
\begin{equation*}
   \rho_0:=\{p\in\pi(n) \mid v_p(o(\zeta))< v_p(n)\}\quad\text{and}\quad \rho_1:=\{p\in\pi(n) \mid v_p(o(\zeta))=v_p(n)\}\,,
\end{equation*}
and for $\rho\subseteq\pi(n)$ set $n(\rho):=N_{\rho'}\cdot\und n_\rho$.
Then
\begin{equation}\label{eqn zeta equation}
   \sum_{\rho\subseteq \pi(n)} (-1)^{|\rho|} \sum_{k\in(\ZZ/t\ZZ)^\times} \zeta^{kn(\rho)} \ =
   \begin{cases}
      0\,, & \text{if $\rho_0\neq \emptyset$,} \\
      \sum\limits_{\rho\subseteq \rho_1} \frac{\phi(t)}{\prod\limits_{p\in \rho}(p-1)}\,,& \text{if $\rho_0=\emptyset$.}
   \end{cases}
\end{equation}
In particular, the number in (\ref{eqn zeta equation}) is a non-negative integer and it is equal to 0 if and only if $\rho_0\neq\emptyset$. 
\end{lem}

\begin{proof}
For any $\rho\subseteq\pi(n)$, Lemma~\ref{lem trace} implies that
\begin{equation}\label{eqn a_I}
   a_\rho:=\sum_{k\in(\ZZ/t\ZZ)^\times} \zeta^{kn(\rho)} = \Tr_{\QQ_t/\QQ}(\zeta^{n(\rho)}) = 
   \mu(o(\zeta^{n(\rho)}))\frac{\phi(t)}{\phi(o(\zeta^{n(\rho)}))}\,.
\end{equation}
Note that the latter term depends only on $o(\zeta^{n(\rho)})$. For any $\rho\subseteq \pi(n)$ and any $q\in\pi(N)$, one has
\begin{equation}\label{eqn 3 cases}
   v_q(o(\zeta^{n(\rho)})) =
   \begin{cases}
      0\,, & \text{if $q\in\pi(N)\smallsetminus \rho$,} \\
      0\,, & \text{if $q\in \rho\cap \rho_0$,} \\
      v_q(o(\zeta)) - v_q(n) + 1\,, & \text{if $q \in \rho\smallsetminus \rho_0$.}
   \end{cases}
\end{equation}
In particular, $v_q(o(\zeta^{n(\rho)})) = 0$ for all $q\in \rho_0$, independent of $\rho$. This implies
\begin{equation*}
   o(\zeta^{n(\rho)}) = o(\zeta^{n(\rho\cup \rho_0)})\quad \text{and}\quad a_{\rho}=a_{\rho\cup \rho_0}
\end{equation*}
for all $\rho\subseteq \pi(n)$. Therefore, the left hand side of Equation~(\ref{eqn zeta equation}) equals
\begin{equation*}
   \sum_{\rho_0\subseteq \sigma\subseteq\pi(n)} \sum_{\substack{\rho\subseteq \sigma \\ \rho\cup \rho_0 = \sigma}} (-1)^{|\rho|} \ a_\sigma\,.
\end{equation*}
However, a subset $\rho$ of $\sigma$ satisfies $\rho\cup \rho_0=\sigma$ if and only if $\sigma\smallsetminus \rho_0\subseteq \rho$. Therefore, the last double sum is equal to
\begin{equation*}
   \sum_{\rho_0\subseteq \sigma\subseteq\pi(n)} 
   \Bigl(\sum_{\sigma\smallsetminus \rho_0 \subseteq \rho \subseteq \sigma} (-1)^{|\rho|}\Bigr) \ a_\sigma\,.
\end{equation*}
If $\rho_0\neq \emptyset$ then $\sum_{\sigma\smallsetminus \rho_0 \subseteq \rho \subseteq \sigma} (-1)^{|\rho|} = 0$ 
for all $\rho_0\subseteq \sigma\subseteq \pi(n)$, so that the left hand side of Equation~(\ref{eqn zeta equation}) equals $0$ 
in this case, as claimed in the lemma.

So, we assume from now on that $\rho_0=\emptyset$. In this case, Equation~(\ref{eqn 3 cases}) becomes
\begin{equation}\label{eqn 2 cases}
   v_q(o(\zeta^{n(\rho)})) =
   \begin{cases}
      0\,, & \text{if $q\in\pi(N)\smallsetminus \rho$,} \\
      v_q(o(\zeta)) - v_q(n) + 1>0\,, & \text{if $q\in\rho$,}
   \end{cases}
\end{equation}
for any $\rho\subseteq \pi(n)$ and any $q\in\pi(N)$. 
Now, using Equation~(\ref{eqn a_I}) and that $\mu(o(\zeta^{n(\rho)}))=0$ unless $\rho\subseteq \rho_1$, 
the left hand side of Equation~(\ref{eqn zeta equation}) becomes equal to
\begin{equation*}
   \sum_{\rho\subseteq \rho_1} (-1)^{|\rho|} \mu(o(\zeta^{n(\rho)})) \frac{\phi(t)}{\phi(o(\zeta^{n(\rho)}))}\,.
\end{equation*}
But, by Equation~(\ref{eqn 2 cases}), we have $\mu(o(\zeta^{n(\rho)}))=(-1)^{|\rho|}$ and 
$\phi(o(\zeta^{n(\rho)}))=\prod_{p\in \rho}(p-1)$, for all $\rho\subseteq \rho_1$. This concludes the proof of the lemma.
\end{proof}

\section{Adams operations}\label{sec Adams}

Recall that for any integer $n$, one has a ring homomorphism
\begin{equation*}
   \Psi^n\colon R(G)\to R(G)\quad \text{defined by}\quad  \Psi^n(\chi)(g)=\chi(g^n)\,.
\end{equation*}
Although not obvious, it is a fact that the class function $\Psi^n(\chi)$ is again a virtual character of $G$, see \cite[Exercise~9.3]{Serre}.
By \cite[Theorem~2.1(j)(d)]{Boltje1990}, the canonical Brauer induction formula $a_G$ has the following interesting property with respect
to Adams operations. For any $\chi\in R(G)$ and any $n\in\ZZ$, one has
\begin{equation}\label{eqn canind and Adams}
   \sum_{\substack{[H,\varphi]_G\in G\backslash\calM(G)\\ \varphi^n=1}} \alpha_{[H,\varphi]_G)}(\chi) = (\Psi^n(\chi),1_G)\,.
\end{equation}

We fix now a finite group $G$ and a character $\chi$ of $G$. Let $\Lambda$ denote the set of all positive divisors of $\exp(G)$, the exponent of $G$. We define a function $f\colon \Lambda\to\ZZ$ by
\begin{equation*}
   f(n):= \sum_{\substack{[H,\varphi]_G\in G\backslash\calM(G) \\ o(\varphi)=n}} \alpha_{[H,\varphi]_G}(\chi)
\end{equation*}
and define $f_+\colon\Lambda \to \ZZ$ and $f^+\colon\Lambda\to \ZZ$ as in Lemma~\ref{lem summatory}. Then, using (\ref{eqn canind and Adams}), we have
\begin{equation*}
   f_+(n) = \sum_{\substack{[H,\varphi]_G\in G\backslash\calM(G)\\ \varphi^n=1}} \alpha_{[H,\varphi]_G)}(\chi) = (\Psi^n(\chi),1_G)
\end{equation*}
and
\begin{equation*}
   f^+(n)= \sum_{\substack{[H,\varphi]_G\in G\backslash\calM(G)\\ n \mid o(\varphi)}} \alpha_{[H,\varphi]_G)}(\chi)\,,
\end{equation*}
for all $n\in\Lambda$. Note that $f^+(n)= S(G,\chi,n)$, the function we defined in (\ref{eqn S function}).
Lemma~\ref{lem summatory} now immediately implies Part~(a) of the following theorem, which is the Theorem~B from the introduction.

\begin{thm}\label{thm nonnegative}
Let $G$ be a finite group, let $\chi$ be a character of $G$ afforded by the representation $\scrX\colon G\to \GL_d(\CC)$, and let $n$ be a positive divisor of $\exp(G)$.

\smallskip
{\rm (a)} One has
\begin{equation*}
   S(G,\chi,n) = \sum_{\substack{[H,\varphi]_G\in G\backslash\calM(G)\\ n \mid o(\varphi)}} \alpha_{[H,\varphi]_G)}(\chi) =
   \sum_{\rho\subseteq\pi(n)} (-1)^{|\rho|} \bigl(\Psi^{n(\rho)}(\chi),1\bigr)\,.
\end{equation*}
Here, for $\rho\subseteq \pi(n)$, the number $n(\rho)$ is defined as $n(\rho)=\exp(G)_{\rho'}\cdot \und{n}_\rho$, see Lemma~\ref{lem summatory}.

\smallskip
{\rm (b)} One has $S(G,\chi,n)\ge 0$. 

\smallskip
{\rm (c)} The following are equivalent:

\smallskip
\qquad {\rm (i)} There exists $(H,\varphi)\in\calM(G,\chi)$ such that $n= o(\varphi)$.

\smallskip
\qquad {\rm (ii)} One has $S(G,\chi,n)> 0$.

\smallskip
\qquad {\rm (iii)} There exists an element $x\in G$ and an eigenvalue $\zeta$ of $\scrX(x)$ such that $n= o(\zeta)$.
\end{thm}

\begin{proof}
As mentioned above, we only need to prove (b) and (c). Let $\calR\subseteq G$ denote a set of representatives of the equivalence relation $x\sim y:\iff \langle x \rangle = \langle y \rangle$. Moreover, for each $x\in\calR$, let $\zeta_{x,1},\ldots,\zeta_{x,d}$ denote the eigenvalues of $\scrX(x)$. Then,
 \begin{equation}\label{eqn S computation}
 \begin{split}
    & |G|\cdot \sum_{\rho\subseteq \pi(n)} (-1)^{|\rho|}\Bigl(\Psi^{n(\rho)}(\chi)\ ,\ 1_G\Bigr)
    = \sum_{\substack{\rho\subseteq \pi(n) \\ x\in G}} (-1)^{|\rho|}\Psi^{n(\rho)}(\chi)(x) 
    =  \sum_{\substack{\rho\subseteq \pi(n) \\ x\in G}} (-1)^{|\rho|} \chi(x^{n(\rho)}) \\
    & \qquad\qquad =  \sum_{\substack{\rho\subseteq \pi(n) \\ x\in G}} (-1)^{|\rho|}  \sum_{l=1}^d \zeta_{x,l}^{n(\rho)} 
    =  \sum_{\substack{\rho\subseteq \pi(n) \\ x\in \calR \\ l\in \{1,\ldots,d\}}} (-1)^{|\rho|} \sum_{k\in(\ZZ/o(x)\ZZ)^\times }\zeta_{x,l}^{n(\rho)k}
    = \sum_{\substack{x\in\calR\\ l\in\{1,\ldots,d\}}} a_{x,l}\,,
\end{split}
\end{equation}
where
\begin{equation}\label{eqn zeta sum e}
   a_{x,l}=\sum_{\rho\subseteq\pi(n)} (-1)^{|\rho|} \sum_{k\in(\ZZ/o(x)\ZZ)^\times} \zeta_{x,l}^{kn(\rho)}\,.
\end{equation}
Note that $a_{x,l}$ is exactly the expression analyzed in Lemma~\ref{lem technical}, with $N=\exp(G)$ and $t=o(x)$. 
Since $a_{x,l}$ is a non-negative integer for every $x\in\calR$ and $l\in\{1,\ldots,d\}$, 
by Lemma~\ref{lem technical}, also $S(G,\chi,n)\ge 0$ and (b) is proved.

\smallskip
Next we prove (c). Clearly, if $S(G,\chi,n)>0$ then there exists a pair $(H,\varphi)\in\calMtilde(G,\chi)\subseteq\calM(G,\chi)$ with $n\mid o(\varphi)$. Now Remark~\ref{rem reformulation} shows that (ii) implies (i).

Next suppose that (i) holds. Then, by Proposition~\ref{prop 5 equivalences} and Remark~\ref{rem reformulation}, there also exists $(H,\varphi)\in\calM(G,\chi)$ with $H$ cyclic and $n = o(\varphi)$. Let $x\in\calR$ be a generator of $H$. Then, since $\varphi$ is a constituent of $\chi_H$ and $n= o(\varphi)$, the matrix $\scrX(x)$ has an eigenvalue $\zeta_{x,l}$ whose order is equal to $n$. This shows that (i) implies (iii).

Finally suppose that (iii) holds, i.e. that $n = o(\zeta_{x,l})$ for some $x\in\calR$ and some $l\in\{1,\ldots,d\}$.
Then the set $\rho(x,l)_0=\{p\in\pi(n)\mid v_p(o(\zeta_{x,l}))<v_p(n)\}$ from Lemma~\ref{lem technical} associated to $\zeta_{x,l}$ is empty, implying that $a_{x,l}>0$. Now, Equation~(\ref{eqn S computation}) 
implies that $S(G,\chi,n)>0$. Thus, (iii) implies (ii).
\end{proof}

For $\chi\in\Irr(G)$ we set $F(G,\chi):=S(G,\chi,c(\chi))$. By the above theorem, this is a non-negative integer. The above theorem immediately implies, that the number $F(G,\chi)$ is an indicator for the validity of Conjecture A. More precisely, one has:

\begin{cor}\label{cor main}
Let $G$ be a finite group and let $\chi$ be an irreducible character of $G$. Then Conjecture~A holds for $(G,\chi)$ if and only if $F(G,\chi)>0$.
\end{cor}

\begin{rem}\label{rem main thm}
(a) One could have defined $S(G,\chi,n)$ directly as the integer $\sum_{\rho\subseteq\pi(n)} (-1)^{|\rho|} \bigl(\Psi^{n(\rho)}(\chi),1\bigr)$, 
without using the canonical Brauer induction formula. 
Theorem~\ref{thm nonnegative} could have been proved by just using Lemmas~\ref{lem summatory} and \ref{lem technical}. 
But we found the connection with the canonical Brauer induction formula interesting enough to include it in the paper.

\smallskip
(b) Conjecture~A has been proved for all solvable groups in \cite{BKNT}.  Using \cite{GAP}, it has also been verified for all sporadic simple groups and their irreducible characters through the positivity of the number $F(G,\chi)$ .

\smallskip
(c) The criterion that $\sum_{\rho\subseteq\pi(n)} (-1)^{|\rho|} \bigl(\Psi^{n(\rho)}(\chi),1\bigr)>0$ for $n=c(\chi)$ can be verified very quickly
with \cite{GAP}, even for very large groups, as long as one knows the character table of $G$ together with the power maps between the conjugacy classes.
\end{rem}

\begin{examples}

We evaluate $S(G,\chi,n)$ for various special values of characters $\chi$ of $G$ and positive divisors $n$ of $\exp(G)$.

\smallskip

(a) If $\chi=1_G$ is the trivial character then $\Psi^k(1_G)=1_G$ for any integer $k$. Therefore,
\begin{equation*}
   S(G,1_G,n)= \sum_{\rho\subseteq\pi(n)}(-1)^{|\rho|} \Bigl(\Psi^{n(\rho)}(1_G),1_G\Bigr) = \sum_{\rho\subseteq\pi(n)}(-1)^{|\rho|}  = (1-1)^{|\pi(n)|} =
   \begin{cases} $1$, & \text{if $n=1$;} \\ $0$, & \text{if $n>1$.} \end{cases}
\end{equation*}

\smallskip

(b) More generally, if $\chi$ is a linear character $\varphi$ of $G$ then its canonical Brauer induction formula is given as $a_G(\varphi)=[G,\varphi]_G$. By the definition of $S(G,\chi,n)$ in (\ref{eqn S def}), we obtain
\begin{equation*} 
   S(G,\varphi,n)=\begin{cases} 1, & \text{if $n\mid o(\varphi)$;} \\ 0, & \text{otherwise.}    \end{cases}
\end{equation*}   
This also reconfirms the result in (a).

\smallskip

(c) Next we evaluate $S(G,\chi,n)$ for the regular character $\chireg$ of $G$. For each positive divisor $k$ of $\exp(G)$ let $f(k)$ denote the number of elements of $G$ of order $k$. Then $f_+(k)$ (resp.~$f^+(k)$) is the number of elements of $G$ of order dividing $k$ (resp.~whose order is a multiple of $k$). Note that $\Psi^k(\chireg)=|G|\chi_k$, where $\chi_k$ is the characteristic function on the set of all elements $g\in G$ with $g^k=1$, or equivalently with $o(g)\mid k$. Thus, using Lemma~\ref{lem summatory}, we obtain
\begin{align*}
   S(G,\chireg,n) & = \sum_{\rho\subseteq\pi(n)} (-1)^{|\rho|} \Bigl(\Psi^{n(\rho)}(\chireg),1_G\Bigr)
   = \sum_{\rho\subseteq\pi(n)} (-1)^{|\rho|} |\{g\in G\mid o(g)\mid n(\rho)\}| \\
   & = \sum_{\rho\subseteq\pi(n)} (-1)^{|\rho|} f_+(n(\rho))
   = f^+(n) = |\{g\in G\mid n\mid o(g)\}|\,.
\end{align*}
In particular, $S(G,\chireg,n)>0$ if and only if $G$ has an element of order $n$.

\smallskip
(d) Now suppose that $n=1$. Then $\pi(n)=\emptyset$, $n(\emptyset)=\exp(G)$ and $\Psi^{\exp(G)}(\chi)= \chi(1)\cdot 1_G$. Thus, 
\begin{equation*}
   S(G,\chi,1)=\bigl(\Psi^{\exp(G)}(\chi),1_G\bigr) = \chi(1)\,.
\end{equation*}

\smallskip
(e) Suppose that $n=\exp(G)$. If $G$ has no element of order $\exp(G)$ then, by Theorem~\ref{thm nonnegative}(c) using the equivalence of (ii) and (iii), or also by the definition in (\ref{eqn S def}), one has $S(G,\chi,\exp(G))=0$. However, if $G$ has an element of order $\exp(G)$ we do not have a general formula for $S(G,\chi, \exp(G))$.

\smallskip

(f) One might also ask if the virtual character $\sum_{\rho\subseteq\pi(n)} (-1)^{|\rho|} \Psi^{n(\rho)}(\chi)$ is actually a genuine character of $G$. The answer is negative as the following trivial example shows. Let $G$ be a cyclic group of prime order $p$, let $\chi\in\Irr(G)$ be faithful and $n=p$. Then the above virtual character is equal to $1_G-\chi$.
\end{examples}




\end{document}